\documentclass[twoside,12pt,leqno]{amsproc}
\usepackage{amssymb,latexsym,enumerate,verbatim,stmaryrd,tikz}
\usepackage[pagebackref,colorlinks,linkcolor=purple,citecolor=blue,urlcolor=blue
,hypertexnames=true]{hyperref}
\usepackage{amsrefs}    
\usepackage{etoolbox}   
\hypersetup{citecolor=purple, linkcolor=blue, colorlinks=true}
\numberwithin{table}{section}

\theoremstyle{plain}
\newtheorem{theorem}{Theorem}[section]
\newtheorem{lemma}[theorem]{Lemma}

\newtheorem{question}[theorem]{Problem}

\theoremstyle{definition} 
\newtheorem{definition}[theorem]{Definition}
\newtheorem{remark}[theorem]{Remark}
\newtheorem{example}[theorem]{Example}
\AtEndEnvironment{remark}{\null\hfill$\triangle$}

\newcommand{\lhdeq}{\trianglelefteqslant}    

\newcommand{\C}{\mathsf{C}}
\newcommand{\CC}{\mathcal{C}}
\newcommand{\D}{\mathsf{D}}
\newcommand{\Z}{\mathbb{Z}}

\newcommand{\alt}{\textup{\sf A}}
\newcommand{\Aut}{\textup{Aut}}

\newcommand{\F}{\mathbb{F}}
\newcommand{\GL}{\mathrm{GL}}
\newcommand{\GaL}{\Gamma\mathrm{L}}
\newcommand{\GSp}{\mathrm{GSp}}
\newcommand{\GU}{\mathrm{GU}}

\newcommand{\PSL}{\mathrm{PSL}}
\newcommand{\PGL}{\mathrm{PGL}}
\newcommand{\SL}{\mathrm{SL}}
\newcommand{\SU}{\mathrm{SU}}

\newcommand{\Out}{\textup{Out}}
\newcommand{\SO}{\mathrm{SO}}
\newcommand{\Sp}{\mathrm{Sp}}
\newcommand{\sym}{\textup{\sf S}}
\newcommand{\Sym}{\textup{\rm Sym}}

\newcommand{\eoc}{{\hfill $\Box$}} 

\makeatletter        
\def\@adminfootnotes{%
  \let\@makefnmark\relax  \let\@thefnmark\relax
  \ifx\@empty\@date\else \@footnotetext{\@setdate}\fi
  \ifx\@empty\@subjclass\else \@footnotetext{\@setsubjclass}\fi
  \ifx\@empty\@keywords\else \@footnotetext{\@setkeywords}\fi
  \ifx\@empty\thankses\else \@footnotetext{%
    \def\par{\let\par\@par}\@setthanks}%
  \fi}\makeatother   

\begin{document}

\hyphenation{}

\title[Composition length of primitive groups and completely reducible groups]{Bounding the composition length of\\ primitive permutation groups and\\ completely reducible linear groups}
\author{S.\,P. Glasby, Cheryl E. Praeger, Kyle Rosa, Gabriel Verret}

\address[Glasby, Praeger, Rosa]{
Centre for Mathematics of Symmetry and Computation,
University of Western Australia,
35 Stirling Highway,
Crawley 6009, Australia.\newline
 Email: {\tt Stephen.Glasby@uwa.edu.au} URL: {\tt www.maths.uwa.edu.au/$\sim$glasby/}\newline
 Email: {\tt Cheryl.Praeger@uwa.edu.au}; URL: {\tt www.maths.uwa.edu.au/$\sim$praeger}\newline
 Email: {\tt Kyle.Rosa@research.uwa.edu.au}
 }
 
\address[Verret]{Department of Mathematics, The University of Auckland\\
Private Bag 92019, Auckland 1142, New Zealand. Email: {\tt g.verret@auckland.ac.nz}}

\date{\today}

\begin{abstract}
We obtain upper bounds on the composition length of a finite permutation group in terms of the degree and the number of orbits, and analogous bounds  for primitive, quasiprimitive and semiprimitive groups. Similarly, we obtain upper bounds on the composition length of a finite completely reducible linear group in terms of some of its parameters. In almost all cases we show that the bounds are sharp, and describe the extremal examples. 
\end{abstract}

\maketitle
\begin{center}{{\sc\tiny MSC 2010 Classification: 20B15, 20H30, 20B05}}\end{center}

\section{Introduction}\label{S1}

The composition length  of a finite group  is the length of any composition series of the group. It is sometimes viewed as a measure of its size or complexity. Often it is useful to have bounds in terms of  parameters relevant to the way the group is represented, rather than the abstract group structure. In Subsection~\ref{sub:motiv} we comment on the research questions which motivated our investigation, we describe how our results relate to other work, and mention some open questions.

We obtain upper bounds on the composition length of a finite permutation group in terms of the degree and the number of orbits (Theorem~\ref{C2}), and analogous bounds  for primitive (Theorem~\ref{C1}), quasiprimitive and semiprimitive groups (Theorem~\ref{theo:qpsp}). Similarly, we obtain upper bounds on the composition length of a finite completely reducible linear group in terms of some of its parameters (Theorem~\ref{C3}). We also show in almost all cases that  our bounds are sharp, and describe all extremal examples. For this purpose, we define the following concepts.  A permutation group is \emph{primitive} if it is transitive and  preserves no nontrivial partition of the set on which it acts; transitive permutation groups that preserve some nontrivial point partition are said to be \emph{imprimitive}.

\renewcommand{\labelitemi}{$\circ$}

\begin{definition}\label{maindef}
Let $\sym_4$ denote the symmetric group of degree $4$ in its natural action and let $k$ be a non-negative integer.
\begin{itemize}
\item Let $T_k = \sym_4\wr\cdots\wr\sym_4$, the iterated imprimitive wreath product of $k$ copies of $\sym_4$.
\item Let $P_k=\sym_4\wr T_k$, in its  primitive wreath product action.
\item  Let $L_k=\GL(2,2)\wr T_k$, viewed as an imprimitive linear subgroup of $\GL(2^{2k+1},2)$.
\end{itemize}
Note that $T_k$ is a transitive group of degree $4^k$; in particular $T_0=1$
has degree $1$. Therefore $P_k$ is a primitive group of degree $4^{4^k}$ which
is abstractly isomorphic to $T_{k+1}$.
\end{definition}

For a finite group $G$, let $c(G)$ denote its composition length.

\begin{theorem}\label{C2}
If $G$ is a permutation group of degree $n$ with $r$ orbits, then 
\begin{equation*}
  c(G)\leqslant\frac43\left(n-r\right).
\end{equation*}
Moreover, equality holds if and only if there exist nonnegative integers $k_1,\dots,k_r$ such that  the orbits of $G$ have sizes $4^{k_1},\ldots,4^{k_r}$ and $G$ is permutationally isomorphic to  $T_{k_1}\times\cdots\times T_{k_r}$ in its natural action.
\end{theorem}

\begin{theorem}\label{C1}
If $G$ is a primitive permutation group of degree $n$, then 
\begin{equation*}\label{E1}
  c(G)\leqslant\frac83\log_2 n-\frac43.
\end{equation*}
Moreover, equality holds if and only if  $n=4^{4^k}$ for some $k\geqslant 0$ and $G$ is permutationally isomorphic to $P_k$.
\end{theorem}

These theorems depend on the finite simple group classification since the proof of Theorem~\ref{C1} uses Theorem~\ref{C2}, and the proof of Theorem~\ref{C2} uses an order bound for primitive groups from \cite{M} which depends on the classification.

A group $H$ of linear transformations of a vector space $V$ is \emph{completely reducible} if there is a direct decomposition  $V=V_1\oplus\dots\oplus V_r$, with $r\geqslant 1$, such that each $V_i$ is $H$-invariant and the restriction $H|_{V_i}$ is irreducible.   The $V_i$ are the \emph{irreducible constituents} of $H$.

\begin{theorem}\label{C3}
If  $H$ is a completely reducible subgroup of $\GL(d,p^f)$  with $r$ irreducible constituents $V_1,\dots,V_r$, then
\begin{equation}\label{E3}
  c(H)\leqslant\left(\frac83\log_2 p-1\right)df-r\left(\log_2f+\frac43\right).
\end{equation}
Moreover, equality holds if and only if one of the following occurs:

\begin{enumerate}[{\rm (a)}]
\item  $p^f=2$ and there exist  positive integers $k_1,\dots,k_r$ such that $\dim(V_i)=2^{2k_i+1}$  and $H$ is linearly isomorphic to $L_{k_1}\times\cdots\times L_{k_r}$,  or 
\item $p^f=2^2$, $d=r$ and $H$ is linearly isomorphic to $\GL(1,4)^d\cong (\C_3)^d$.
\end{enumerate}
\end{theorem}

\subsection{Context, discussion, and more results}\label{sub:motiv}

For a finite group $G$ of order $m$, $c(G)\leqslant \log_2(m)$, with equality if and only if $G$ is a $2$-group (with each composition factor cyclic of order $2$).  Similarly each of the upper bounds in \cites{Ba81,Ba82,Cam81,M,PS,Wie69} on the orders of finite primitive permutation groups $G$ of degree $n$ yields an upper  bound for $c(G)$ as a function of $n$. The best of these order bounds  \cite{Cam81}*{Theorem 6.1(S)}, due to Cameron in 1981, depends on the finite simple group classification: namely a primitive group $G$ of degree $n$ is  of affine type, or is in a well understood family of primitive groups of product action type, or satisfies $|G|\leqslant n^{c\log_2\log_2 n}$ for a ``computable constant $c$''.  

In 1993, Pyber~\cite{Pyb93}*{Theorem 2.10} states that, for a primitive permutation group $G$ of degree $n$,  $c(G)\leqslant (2+c)\log_2n$ with $c$ the constant in Cameron's result. A proof of this result appeared recently  in  \cite{GMP}*{Corollary 6.7}.\footnote{The statement in \cite{Pyb93}*{Theorem 2.10}  refers to a paper ``in preparation'' (reference [Py5] in \cite{Pyb93}).}  It has been used in several investigations. For example, it is used for the irreducible case of \cite{LuMM}*{Theorem C}, which bounds the composition length of finite  completely reducible linear groups, and it is used in \cite{DL}*{p.\;305} to bound the invariable generation  number for permutation groups.  For the application in \cite{DL} the result  \cite{LuMM}*{Theorem C} is applied  with the constant $c=2.25$. The paper \cite{GMP} derives many bounds for permutation groups and  linear groups  $G$ focussing on bounds for $|\Out(G)|$. In particular   \cite{GMP}*{Corollary 6.7} yields the bound $c(G)\leqslant (2+c)\log_2n$ with the constant $c= \log_9(48\cdot 24^{1/3})=2.24\cdots$,
that is to say, $c(G)\leqslant c'\log_2n$ with $c' = 4.24\cdots$.

Our investigations began before  \cite{GMP} was published. Because we had been unable to find a proof of  Pyber's result  in the literature,  and because of its diverse applications,  we decided to seek the best value for a constant $c'$ such that $c(G)\leqslant c'\log_2 n$ whenever $G$ is a primitive permutation  group of degree $n$.  Further, we wondered if we could find sharp upper bounds and classify all groups attaining them.  Our Theorem~\ref{C1} achieves this, and in particular  shows that the best value for such a constant $c'$ is $8/3=2.66\cdots$.  

Whereas all the primitive permutation groups $G$ achieving the bounds of Theorem~\ref{C1} are of affine type, the  primitive groups of  degree $n$  covered by Cameron's ``order upper bound''  $n^{c\log_2\log_2 n}$, are in particular not of affine type.   The following companion result to Theorem~\ref{C1} gives a sharp upper bound on the composition length of \emph{non-affine primitive groups},  by which we mean primitive permutation groups with no nontrivial abelian normal subgroups. 

\begin{theorem}\label{theo:na}
If  $G$ is a non-affine primitive permutation group of degree $n$, then  
\begin{equation*}
  c(G)\leqslant{} c_{\rm na}\log_2 n - \frac{4}{3}, \quad \mbox{where}\  c_{\rm na} = \frac{10}{3 \log_25} = 1.43\cdots
\end{equation*}
with equality if and only if $n=5^{4^k}$ and $G= \sym_5\wr T_k$ in product action, for some $k\geqslant0$.
\end{theorem}

We note the striking difference between the logarithmic upper bounds on $c(G)$ for primitive groups $G$  in  Theorems~\ref{C1} and~\ref{theo:na}, and the linear bound for general permutation groups in Theorem~\ref{C2}. 

\begin{question}\label{q1}
Which other infinite families of permutation groups have composition lengths bounded above  by a logarithmic function of the degree? 
\end{question}

Our final main result gives examples of two such families. A permutation group is \emph{quasiprimitive} if each of its nontrivial normal subgroups is transitive. It is \emph{semiprimitive} if each of its normal subgroups is either semiregular or transitive. (A permutation group is \emph{semiregular} if the only element fixing a point is the identity.)

\begin{theorem}\label{theo:qpsp}
  Let $G$ be a permutation group of degree $n$.
  \begin{enumerate}[{\rm (a)}]
    \item If $G$ is quasiprimitive but not primitive, then 
    \[
      c(G)\leqslant{} c_{\rm na}(\log_2{n}-1)-\frac43 = c_{\rm na} \log_2n - 2.76\cdots
    \]
    where $c_{\rm na} = \frac{10}{3 \log_25} = 1.43\cdots$ 
    as in Theorem~\textup{\ref{theo:na}}.
  \item If $G$ is semiprimitive but not quasiprimitive, then
    $c(G)\leqslant{}\frac{8}{3}\log_2{n}-3$.
  \end{enumerate}
\end{theorem}

We give infinitely many examples to show that the bound in Theorem~\ref{theo:qpsp}(b) is  best possible (see Example~\ref{ex:semiprim}). For a semiprimitive group $G$, a normal subgroup of $G$ which is minimal subject to being transitive, is called a \emph{plinth}. If $G$ is a semiprimitive group which achieves the $\frac{8}{3}\log_2{n}-3$ bound  in Theorem~\ref{theo:qpsp}(b), 
then $n$ is a power of $2$ and each plinth of $G$ is a $2$-group (Remark~\ref{rem:semiprim}). Unfortunately  the bound for quasiprimitive groups is not sharp (Remark~\ref{rem:qp}), and we do not even know the best constant $c$ such that $c(G)\leqslant c\log_2n$ for a quasiprimitive group $G$ of degree $n$ which is not primitive. By Theorem~\ref{theo:qpsp}, $c\leqslant c_{\rm na}=1.43\cdots$, and we give examples in Section~\ref{sec:semi} which show that $c\geqslant  \frac{31}{12\log_2 5\, +\, 9\log_2 3}= 0.73\cdots$.

\begin{question}\label{q2}
\begin{enumerate}[{\rm (a)}]
  \item Find a sharp upper bound on the composition length in terms of the
  degree, for quasiprimitive permutation groups which are not primitive.
  \item Determine all semiprimitive groups which achieve the bound in
  Theorem~\textup{\ref{theo:qpsp}(b)}. 
\item For $G\leqslant\sym_n$, with $G$ semiprimitive and not quasiprimitive and
  with
  an insoluble plinth, find a sharp upper bound for $c(G)$ as a function of $n$.
  \end{enumerate}
\end{question}
The proof of  Theorem~\ref{C1}  proceeds by considering various types of finite primitive permutation groups. In particular a  primitive subgroup $G\leqslant \Sym(\Omega) = \sym_n$ may leave invariant a cartesian decomposition $\Omega = \Delta^r$ for  some smaller set $\Delta$ and integer $r\geqslant2$.  In this case $n=m^r$ where $m=|\Delta|$, and the group $G$ is permutationally  isomorphic to a subgroup of the wreath product $\Sym(\Delta)\wr\sym_r$ in product action. Moreover $G$ must project to a transitive subgroup of $\sym_r$, and for $c(G)$ to be maximised we require the composition length of this transitive subgroup of $\sym_r$ to be as large as possible.  In other words, in order to prove Theorem~\ref{C1} for these product action primitive groups we need the bound (and extreme examples) from Theorem~\ref{C2} for transitive groups. We note that our result Theorem~\ref{C2} extends early work by Fisher dating from 1974. Namely we improve~\cite{F}*{Lemma~2} by proving that permutation groups of the form $T_{k_1}\times \cdots\times T_{k_r}$ are the only examples, with $r$ orbits, for which equality occurs in the upper bound in Theorem~\ref{C2}.  (One reason for giving an independent proof  is that there appears to be a small error in the proof of~\cite{F}*{Lemma~2}: the sentence beginning ``If $G$ is transitive and imprimitive'' is incorrect.)  

Another class of primitive groups which must be considered when proving Theorem~\ref{C1} are those of affine type. These are groups of affine transformations of a finite vector space and have the form $N\rtimes G_0$, where $N$ is the group of translations, and $G_0$ is an irreducible subgroup of linear transformations. Thus in order to prove Theorem~\ref{C1} for affine primitive groups we need the bound (and extreme examples) from Theorem~\ref{C3} for irreducible groups.

Our work on completely reducible groups also strengthens various results in the literature. As early as 1974, Fisher~\cites{F,F2} obtained estimates for the polycyclic  chief lengths of linear groups (over arbitrary fields).    More recent work by Lucchini et al. \cite{LuMM}*{Theorem C}, relying on the finite simple group classification,  shows that, for a completely reducible subgroup $G\leqslant\GL(d, p^f)$ (with $p$ prime),  $c(G)\leqslant c_{\rm cr} (\log_2 p)dnf$ for some constant $c_{\rm cr}$.   Theorem~\ref{C3}  shows that the best possible constant $c_{\rm cr}$ is $8/3$. The immediate motivation for our work was~\cite{CompFactors}*{Theorem~1} (on the number of composition 
 factors $\C_p$)  which suggested that it  might be possible to find sharp upper bounds  on $c(G)$ for all finite  completely reducible groups.

\section{Preliminaries}

We say that $H$ is a~\emph{subdirect subgroup} of $G_1\times\cdots\times G_r$ if $H$ projects onto each direct factor.  Given a group $G_1$ and a transitive permutation group $G_2$ of degree $r$, the \emph{wreath product}  $G_1\wr G_2$ is $B \rtimes G_2$ where $B=B_1\times\cdots\times B_r\cong G_1^r$, with  $G_2$  acting naturally by conjugation  on the $B_i$.  We say that $H$ is a~\emph{subwreath subgroup} of $G_1\wr G_2$  if $H$ projects onto $G_2$, and the normaliser in $H$ of $B_1$ projects onto $B_1$.

\begin{lemma}\label{La}
Let $G$ be a finite group.
\begin{enumerate}[{\rm (a)}]
\item  If $N\trianglelefteqslant G$, then $c(N)\leqslant c(G)$ with equality if and only if $N=G$.
\item  If $H$  is a subdirect subgroup of $G_1\times\cdots\times G_r$, then $c(H)\leqslant \sum_{i=1}^r c(G_i)$, with equality if and only if $H=G_1\times\cdots\times G_r$.
\item  If $G_2$ is a transitive permutation group of degree $r$ and $H$ is a subwreath subgroup of $G_1\wr G_2$, then $c(H)\leqslant r\cdot c(G_1) + c(G_2)$, with equality if and only if $H=G_1\wr G_2$.
\end{enumerate}
\end{lemma}

\begin{proof}
(a) Clearly $c(G)=c(N)+c(G/N)$ and $c(G/N)=0$ if and only if $N=G$.

(b) Let $H_0=H$ and, for $1\leqslant i\leqslant r$, let $K_i$ be the kernel of the projection map $\pi_i\colon H\to G_i$, and $H_i= H\cap K_1\cap\cdots\cap K_i$. The normal series $H=H_0\trianglerighteqslant H_1\trianglerighteqslant\cdots\trianglerighteqslant H_r=1$ has factor groups
\begin{equation*}
  \frac{H_{i-1}}{H_i}=\frac{H_{i-1}}{H_{i-1}\cap K_i}
  \cong\frac{H_{i-1}K_i}{K_i}\trianglelefteqslant\frac{H}{K_i}\cong G_i.
\end{equation*}
Therefore $c(H_{i-1}/H_i)\leqslant c(G_i)$ by part~(a), and so
\begin{equation*}
  c(H)=\sum_{i=1}^rc(H_{i-1}/H_i) \leqslant \sum_{i=1}^rc(G_i)=c(G_1\times\cdots\times G_r).
\end{equation*}
If equality holds, then for each $i$, $c(H_{i-1}/H_i)= c(G_i)$ which implies that $H_{i-1}/H_i\cong G_i$ by part~(a). In particular, $|H_{i-1}/H_i|=|G_i|$ and so $|H|=\prod_{i=1}^r |H_{i-1}/H_i|=\prod_{i=1}^r |G_i|=|G|$ and thus $H=G$, as claimed.

(c) Write $G_1\wr G_2 = B\rtimes G_2$ where $B=B_1\times\dots\times B_r\cong G_1^r$ and $G_2$ permutes the $B_i$ transitively by conjugation. Let $K= H\cap B$ and let $N$ be the normaliser of $B_1$ in $G_1\wr G_2$. Since $H$ is a subwreath subgroup of $G_1\wr G_2$, we have $H/K\cong G_2$ and $H\cap N$ projects onto $B_1$.    In particular,
\begin{equation*}
  c(H) = c(K) + c(G/K) =  c(K) + c(G_2).
\end{equation*}
For each $i$, let $\pi_i\colon B\to B_i$ be the natural projection map and let $K_i= \pi_i(K)$. Since $B\lhdeq N$, we see $K=H\cap B\lhdeq H\cap N$
and therefore $\pi_1(K)\lhdeq \pi_1(H\cap N)$, that is $K_1\lhdeq B_1$.
 Since $G_2$ is transitive on $\{B_1,\dots,B_r\}$, we have that $K_i\lhdeq B_i$ for each~$i$. Hence, part~(a) implies $c(K_i)\leqslant c(B_i)=c(G_1)$ for each~$i$.

 However, $K$ is a subdirect subgroup of $K_1\times \dots\times K_r$
 by the definition of $K_i$. Therefore by part~(b), $c(K)\leqslant \sum_{i=1}^r c(K_i)  = r\cdot c(K_1) \leqslant r\cdot c(G_1)$. Thus $c(H)\leqslant r\cdot c(G_1) + c(G_2)$.
 
 We see that equality occurs only if $c(K_i)= c(B_i)$, and
 hence $K_i = B_i$, for each $i$. Thus $K$ is a subdirect subgroup of $B=B_1\times\dots\times B_r$, and $c(K) =  r\cdot c(G_1) = \sum_{i=1}^r c(B_i)$. This implies that $K=B$ by part (b), and hence that $H=G_1\wr G_2$, as desired.   
\end{proof}

\begin{remark}
Intransitive permutation groups give rise to subdirect subgroups, and imprimitive permutation groups give rise to subwreath subgroups.
\end{remark}

We use the following order bounds, from  \cites{AG, K}, on the outer automorphism group $\Out(T)$ of a nonabelian simple group $T$. 

\begin{lemma}\label{lem:out}
Let $T$ be a finite nonabelian simple group, and suppose that $T$ has a proper subgroup of index $m$. Then
\begin{enumerate}[{\rm (a)}]
  \item  either $|\Out(T)|< m/2$, or $(T, m, |\Out(T)|) = ( \alt_6, 6, 4)$, and
  \item $|\Out(T)|\leqslant\log_2|T|$.
\end{enumerate}
\end{lemma}

\begin{proof}
If $T=\alt_6$ then $ |\Out(T)|=4$, and either $m=6$ or $m\geqslant10$. Thus part (a) holds for $\alt_6$. If $T\ne \alt_6$, then  by \cite{AG}*{Lemma 2.7(i)},  $|\Out(T)|< m/2$, so part (a) is proved. Part (b) is proved in~\cite{K}.
\end{proof}

\section{Proof of Theorem~\ref{C2}}\label{S5}

\begin{proof}[Proof of Theorem~\ref{C2}]
Let $G$ be a permutation group of degree $n$ with $r$ orbits. The proof is by induction on $n$. It is easy to check that the result holds for $n=1$.

Suppose first that $G$ is intransitive, that is $r\geqslant 2$. Let $\Omega_1,\dots,\Omega_r$ be the $G$-orbits and let $n_i=|\Omega_i|$ for each $i$.  Let $G_i$ be the permutation group induced by $G$  on $\Omega_i$. By induction, $c(G_i)\leqslant \frac43(n_i-1)$. Since $G$ is a subdirect subgroup of $G_1\times\cdots\times G_r$, it follows from Lemma~\ref{La}(b) and induction that
\begin{equation*}
  c(G)\leqslant\sum_{i=1}^r c(G_i)\leqslant\sum_{i=1}^r \frac43(n_i-1)=\frac43(n-r),
\end{equation*}
with equality only if $G=G_1\times\cdots\times G_r$ and $c(G_i)=\frac43(n_i-1)$ for each $i$. By induction,  $G_i=T_{k_i}$ for some $k_i$ satisfying $n_i=4^{k_i}$ and thus $G=T_{k_1}\times\cdots\times T_{k_r}$, as desired.

We may thus assume that $G$ is transitive. Suppose that $G$ is imprimitive and preserves a block system $\mathcal{B}:=\{B_1,\dots, B_s\}$,  where $1<s<n$. Let $G_2$ be the (transitive) permutation group  induced by $G$ on $\mathcal{B}$, and let $G_1$ be the (transitive) permutation group induced on $B_1$ by the setwise stabiliser in $G$ of $B_1$.  Then $G$ is a subwreath subgroup of $G_1\wr G_2$ and hence, by Lemma~\ref{La}(c), $c(G)\leqslant s\cdot c(G_1) + c(G_2)$. Since $G_1$, $G_2$ are transitive permutation groups of degree $n/s$ and~$s$, respectively,  it follows by induction that
\begin{equation*}
  c(G)\leqslant s\cdot c(G_1)+c(G_2)
  \leqslant\frac{4s}3\left(\frac{n}{s}-1\right)+\frac43\left(s-1\right)
  =\frac43\left(n-1\right),
\end{equation*}
with equality only if $G=G_1\wr G_2$, $c(G_1)=\frac43(\frac{n}{s}-1)$ and $c(G_2)=\frac43(s-1)$. By induction, this implies $G_1=T_{k_1}$ and $G_2=T_{k_2}$ for some integers $k_1$ and $k_2$ and thus $G=T_{k_1}\wr T_{k_2}=T_{k_1+k_2}$.

Finally, we assume that $G$ is primitive. We used a database of primitive groups of small degree (see~\cite{CQR-D}) to check that the bound is satisfied when $n\leqslant 24$ and equality holds only for $T_1=\sym_4$. We thus assume that $n\geqslant 25$. If $G$ contains the alternating group of degree $n$, then $c(G)\leqslant 2$ and again the result holds. We may thus assume that this is not the case and, by~\cite{M}*{Corollary~1.4} we have $|G|\leqslant 2^{n-1}$. This implies that $c(G)\leqslant\log_2(2^{n-1})=n-1<\frac43(n-1)$. This completes the proof of Theorem~\ref{C2}.
\end{proof}

\section{Proof of Theorem~\ref{C3}}\label{S6}

\begin{proof}[Proof of Theorem~\ref{C3}] 
Let $H\leqslant \GL(d,p^f)$, such that $H$ is completely reducible on $V=\mathbb{F}_{p^f}^d$. We fix the prime $p$ and use induction on pairs $(d,f)$ which are ordered lexicographically, where $(d_1,f_1)<(d_2,f_2)$ means $d_1<d_2$, or $d_1=d_2$ and $f_1<f_2$.  The case $d=1$ below will serve as the base of our induction. 

{\sc Case 0. $d=1$.}\quad As $\GL(1,p^f)\cong\C_{p^f-1}$ is cyclic, we have $c(H)\leqslant c(\GL(1,p^f))$. Here $d=r=1$ so it suffices to show that $c(\GL(1,p^f))\leqslant \left(\frac83\log_2 p-1\right)f-(\log_2f+\frac43)$  with equality if and only if $p=f=2$. Suppose first that $p=2$. The claim is easily verified for $f\leqslant 3$.  For $f\geqslant 4$, as $2^f-1$ is odd, we have 
\[
  c(\C_{2^f-1})\leqslant\log_3(2^f-1) < f\log_3 2<\frac53f-\log_2f-\frac43.
\]
We may thus assume that $p\geqslant 3$. One can check that, for all positive $f$, we have $\log_2f+\frac43\leqslant \left(\frac53\log_2 p-1\right)f$ and thus 
\begin{equation*}
c(\C_{p^f-1})\leqslant\log_2(p^f-1)< f \log_2 p\leqslant \left(\frac83\log_2 p-1\right)f-\left(\log_2f+\frac43\right).
\end{equation*}
This completes the proof of the case $d=1$. From now on, we assume that $d\geqslant 2$. \eoc

We divide the proof into cases
mirroring Aschbacher's classification of finite linear groups \cite{Asch}
into nine classes $\CC_1,\dots,\CC_9$. The end of a case will be denoted by $\Box$.

{\sc Case 1. $H\in\CC_1$.}\quad 
Here $H$ is reducible. As $H$ is completely reducible, it leaves invariant
a direct decomposition $V=V_1\oplus \dots\oplus V_r$ with $H$ acting
irreducibly on each of the $V_i$, and $r\geqslant 2$. Let $d_i=\dim(V_i)$
and $H_i=H|_{V_i}$. Note that $H_i$ is an irreducible subgroup of $\GL(V_i)$ and $H$ is a subdirect subgroup of $H_1\times H_2\times\dots\times H_r$. By induction, 
\begin{equation*}
 c(H_i)\leqslant\left(\frac83\log_2 p-1\right)d_i f-\left(\log_2f+\frac43\right)
\end{equation*}
for each $i$. Since $\sum_{i=1}^r d_i=d$,  Lemma~\ref{La}(b) implies $c(H)\leqslant \sum_{i=1}^r c(H_i)$ and so
\begin{equation*}
c(H)\leqslant \left(\frac83\log_2 p-1\right)d f- r\left(\log_2f+\frac43\right),
\end{equation*}
with equality if and only if $H = H_1\times\dots\times H_r$ and $c(H_i)=\left(\frac83\log_2 p-1\right)d_i f-\left(\log_2f+\frac43\right)$ for each $i$. By induction, this occurs if and only if either $p^f=2$ and each $H_i$ equals $L_{k_i}$ for some $k_i$, or $p^f=2^2$ and each $H_i$ equals $\GL(1,4)$. Since the value of $p^f$ is independent of $i$, equality holds if and only if $H$ is as in Theorem~\ref{C3}. \eoc

From now on, we assume that $r=1$, or equivalently, that $H$ is irreducible. 

{\sc Case 2. $H\in\CC_2$.}\quad  Here $H$ is an imprimitive linear group.
That is, $H$ preserves a nontrivial direct decomposition
$V=V_1\oplus\cdots\oplus V_b$, where $d=ab$, $b\geqslant 2$, and $\dim(V_i)=a$
for each $i$. Let $H_2$ be the permutation group induced by the action
of $H$ on $\{ V_1,\dots, V_b\}$ and let $K$ be the kernel of this action.
Note that $H_2$ is transitive. Since $H$ is irreducible, the setwise
stabiliser of $V_1$ in $H$ induces on $V_1$ an irreducible subgroup $H_1$
of $\GL(a,p^f)$, and $K|_{V_1}$ is normal in $H_1$. Moreover $H$ is conjugate
to a subwreath subgroup of $H_1\wr H_2$, and so by Lemma~\ref{La}(c),
$c(H)\leqslant b\cdot c(H_1) + c(H_2)$.  Since $a < d$, it follows by induction
that $c(H_1) \leqslant  \left(\frac83\log_2 p-1\right)af- \left(\log_2 f+\frac43\right)$. By Theorem~\ref{C2}, $c(H_2)\leqslant \frac43(b-1)$ hence
\begin{align*}
c(H) \leqslant b \cdot c(H_1) +c(H_2)&\leqslant  \left(\frac83\log_2 p-1\right)df- b\left(\log_2 f+\frac43\right)+\frac43(b-1)\\
     &= \left(\frac83\log_2 p-1\right)df- \left(b\log_2 f +\frac43\right).
\end{align*}
As $r=1$, this expression is less than or equal to the upper bound
in~\eqref{E3}. Suppose now that equality holds. This implies that
$b\log_2 f=\log_2 f$ and thus $f=1$. Equality holding also implies
that $c(H_2)= \frac43(b-1)$ which,  by Theorem~\ref{C2}, implies $H_2=T_{k_2}$
for some $k_2\geqslant1$. Similarly, we must have
$c(H_1) =  \left(\frac83\log_2 p-1\right)af- \left(\log_2 f+\frac43\right)$.
Since $H_1$ is irreducible, it follows by induction that $p^f=2$ and
$H_1=L_{k_1}$ for some $k_1\geqslant0$. Finally, Lemma~\ref{La}(c) implies that
\[
  H=H_1\wr H_2=(\GL(2,2) \wr T_{k_1})\wr T_{k_2}=\GL(2,2) \wr (T_{k_1}\wr T_{k_2})
  =\GL(2,2) \wr T_{k_1+k_2}=L_{k_1+k_2}. \Box
\]

From now on, we assume that $H$ is a primitive linear group.

{\sc Case 3. $H\in\CC_3$.}\quad 
In this case, $H$ preserves on $V$ the structure of a $b$-dimensional
vector space $V'$ over a field of order $p^{fa}$, where $d=ab$ and $a\geqslant 2$.
Note that $H$ is conjugate to a subgroup of
$\GaL(b,p^{fa}) = \GL(b,p^{fa})\rtimes\C_a$.
Let $K= H\cap \GL(b,p^{fa})$. Then $H/K\leqslant \C_a$ and
$c(H)=  c(K)+c(H/K) \leqslant  c(K)+\log_2 a$. By Clifford's
Theorem \cite{LP}*{Theorem 3.6.2},  $K$ acts completely reducibly on $V$,
and by \cite{LP}*{Theorem 1.8.4}, $K$ acts completely reducibly on $V'$,
say with $r'$ irreducible constituents.  Since $b\leqslant d/2$, the inductive
hypothesis yields
\begin{align*}
 c(K)&\leqslant\left(\frac83\log_2 p-1\right)b (fa)- r' \left(\log_2(fa)+\frac43\right)\\
    &=\left(\frac83\log_2 p-1\right)df- r' \left(\log_2 (fa)+\frac43\right)
\end{align*}
and thus
\begin{align*}
 c(H)&\leqslant\left(\frac83\log_2 p-1\right)df- r' \left(\log_2 (fa)+\frac43\right)+\log_2 a\\
 &=\left(\frac83\log_2 p-1\right)df- r' \left(\log_2 f+\frac43\right)- (r'-1)\log_2 a.
\end{align*}
As $r'\geqslant1$, the required inequality \eqref{E3} for $c(H)$ follows
from this. Suppose now that equality holds. It follows that $r'=1$ and
$c(K)=\left(\frac83\log_2 p-1\right)b (fa)- r' \left(\log_2(fa)+\frac43\right)$.
Since $a\geqslant2$ and $b<d$, induction yields that $K=\GL(1,4)$, so $b=1$
and $p^{fa}=2^2$, which implies that $(p,f,a)=(2,1,2)$. Thus $d=ab=2$,
$p^f=2$, $H/K=\C_2$ and $H \cong \GaL(1,4)$ so $H= \GL(2,2) = L_0$.
This concludes the proof in the extension field case. \eoc

We subsequently assume that $H$ preserves no extension field structure on $V$.
Hence $H$ is absolutely irreducible. When \eqref{E3} holds
strictly, as below, equality is impossible.

{\sc Case 4. $H\in\CC_4$.}\quad  Here $H$ is tensor decomposable.
That is, $H$ preserves a decomposition $V=U\otimes W$, where
$a:=\dim(U)\geqslant 2$, $b:=\dim(W)\geqslant 2$, and $d=ab$. We allow $a=b$. Thus
$H\leqslant \GL(U)\circ\GL(W)$, and $H$ projects onto irreducible subgroups
of $H_1\leqslant\GL(U)$ and $H_2\leqslant\GL(W)$.  Hence $H/\Z(H)$ is
a subdirect subgroup $H_1\times H_2\leqslant\GL(a,p^f)\times\GL(b,p^f)$.
By Lemma~\ref{La}(b) we have $c(H) \leqslant c(H_1\times H_2)$. 
It follows by induction that
\begin{align*}
 c(H)&\leqslant \left(\frac83\log_2 p-1\right)af- \left(\log_2 f+\frac43\right) + \left(\frac83\log_2 p-1\right)bf- \left(\log_2 f+\frac43\right)\\
     &= \left(\frac83\log_2 p-1\right)(a+b)f- 2\left(\log_2 f +\frac43\right)\\
 &< \left(\frac83\log_2 p-1\right)(ab)f- \left(\log_2 f +\frac43\right).
 \hspace{70mm}\Box
\end{align*}

Assume now that Case~4 does not apply.
As the $\CC_7$ case is similar to $\CC_4$ case, we treat it
next, and out of order.

{\sc Case 7. $H\in\CC_7$.}\quad Here $H$ is tensor imprimitive
and tensor indecomposable. Therefore $H$ preserves a decomposition
$V=V_1\otimes\dots\otimes V_b$, where $d=a^b$, $a\geqslant 2$, $b\geqslant 2$,
and $\dim(V_i)=a$ for each $i$. Then $H\leqslant K\rtimes\sym_b$,
where $K=\GL(V_1)\circ\dots\circ\GL(V_b)$ contains the scalars
$Z\cong\C_{p^f-1}$ and $K/Z = \PGL(a,p^f)^b$. Since $H$ is not tensor
decomposable, $H/(H\cap K)\cong HK/K$ is a transitive subgroup of $\sym_b$,
and so by Theorem~\ref{C2},   $c(H/(H\cap K))\leqslant \frac43(b-1)$. 
The subgroups $H_i$ of $\PGL(V_i)$ induced by $H\cap K$ are permuted
transitively by $H$. Hence the $H_i$ are irreducible and pairwise
isomorphic. Induction and Lemma~\ref{La}(c) imply
\begin{align*}
c(H) &\leqslant  b\left( \left(\frac83\log_2 p-1\right)af- \log_2 f-\frac43 \right)+\frac43(b-1)\\
     &= \left(\frac83\log_2 p-1\right)(ab)f- \left(b\log_2 f +\frac43\right)\\
     &< \left(\frac83\log_2 p-1\right)(a^b)f- \left(\log_2 f +\frac43\right).
\end{align*}
The final inequality uses the fact that $ab\leqslant a^b$ for $a,b\geqslant2$. In summary, the desired bound holds strictly, when $H$ is tensor imprimitive. \eoc

{\sc Case 5. $H\in\CC_5$.}\quad  Here $H$ is realisable over a proper subfield,
modulo scalars. That is, $f=ab$ with $b\geqslant2$ and we may assume that
$H\leqslant \C_{p^f-1}\circ\GL(d,p^a)$ where the subgroup $\C_{p^f-1}$ of
non-zero $\mathbb{F}_{p^f}$-scalars  meets $\GL(d,p^a)$ in the subgroup
of  $\mathbb{F}_{p^a}$-scalars.  Moreover, $N:=H\cap\GL(d,p^a)$ is normal
in $H$ and
\[
  \frac{H}{N}=\frac{H}{H\cap\GL(d,p^a)}\cong\frac{H\GL(d,p^a)}{\GL(d,p^a)}\leqslant\frac{\C_{p^f-1}\circ\GL(d,p^a)}{\GL(d,p^a)}\cong \frac{\C_{p^f-1}}{\C_{p^a-1}}.
\]
Therefore $|H/N|\leqslant(p^f-1)/(p^a-1)< p^f/(\frac12 p^a)=2p^{f-a}$. Since $N$ is
an irreducible subgroup of $\GL(d,p^a)$ and $a<f$, the inductive
hypothesis gives:
\begin{align*}
  c(H) &= c(N) + c(H/N) < c(N)+\log_2(2p^{f-a})\\
  &\leqslant\left(\frac83\log_2 p-1\right)da-\left(\log_2a+\frac{4}{3}\right)
  +(f-a)\log_2p+1. 
\end{align*}
In order to prove the desired bound~\eqref{E3} with strict inequality,
it suffices to show
\begin{align}\label{E4}
  \log_2f-\log_2a+(f-a)\log_2p+1\leqslant \left(\frac83\log_2 p-1\right)d(f-a).
\end{align}
Since $1\leqslant a\leqslant f/2$ and $2\leqslant d$ we have $f\leqslant d(f-a)$ and
it suffices to show
\begin{align}\label{E5}
  \log_2f+(f-1)\log_2p+1\leqslant\left(\frac83\log_2 p-1\right)f.
\end{align}
This is true when $p=2$ since $\log_2f\leqslant 2f/3$ for $f\geqslant1$.
Suppose now that $p\geqslant3$.
Since $\log_2f\leqslant f-1$ and $2\leqslant \frac53\log_2 p$, we see that $c(H)$ is strictly less than the bound in~\eqref{E3}. \eoc

From now on assume Case 5 does not apply.

{\sc Case 6. $H\in\CC_6$.}\quad  Here $H$ is of symplectic type. Thus
$d=s^a$ where~$s$ is a prime dividing $p^f-1$, and
$H\leqslant \C_{p^f-1}\circ S\ldotp H_0$, where $S$ is an extraspecial group of
order $s^{1+2a}$ whose center $\C_s$ is amalgamated in $\C_{p^f-1}\circ S$,
and $H_0\leqslant\Sp(2a,s)$. By \cite{PSer}*{Table 4},
$|\Sp(2a,s)|\leqslant s^{2a^2+a}$ so $c(H_0)\leqslant  (2a^2+a)\log_2s$
and $c(H)< f\log_2p + 2a + (2a^2+a)\log_2s$.  For convenience
set $z=2a+(2a^2+a)\log_2 s$. We want to show that
\[
  \left(\frac{8}{3}\log_2 p-1\right)s^af-\log_2 f- \frac{4}{3}\geqslant f\log_2 p+z.
\]
Assume to the contrary this does not hold, that is to say,
\begin{equation}\label{eqeq2}
  \left(\frac{8}{3}s^af-f\right)\log_2 p-s^af-\log_2 f- \frac{4}{3}< z.
\end{equation}
Since $\log_2 p\geqslant 1$ and $\frac{4f}{3}\geqslant \log_2 f+\frac{4}{3}$, equation
\eqref{eqeq2} implies
\begin{equation}\label{eqeq1}
  \frac{5}{3}s^af- \frac{7}{3}f< z.
\end{equation}
As $f\geqslant 1$, \eqref{eqeq1} implies $\frac{5}{3}s^a- \frac{7}{3}<  z$
which, in turn, implies $s^a\in\{2, 2^2,2^3,2^4,2^5,3,3^2,5,7\}$.
For fixed $s^a\geqslant 2$ (and thus fixed $z$),  \eqref{eqeq1} is a linear
inequality with only finitely many solutions in $f$. Similarly, for
fixed $s^a$ and $f$, \eqref{eqeq2} is a linear inequality with only
finitely many solutions in $\log_2 p$. It is thus routine to
find all the solutions to \eqref{eqeq2} with $p$ and $s$ primes,
$f\geqslant 1$ and $s^a\geqslant 2$. (This can also easily be automated.)
The solutions are:
\begin{align*}
  (s^a,p^f)\in\{&(2,2),(2,3),(2,2^2),(2,2^3),(3,2),(3,2^2),(3^2,2),\\
  &(2^2,2),(2^2,3),(2^2,2^2),(2^3,2),(2^3,3),(2^3,2^2),(2^4,2),(2^5,2),(5,2)\}.
\end{align*}
Since $s$ divides $p^f-1$, the possibilities reduce to
$(s^a,p^f)\in\{(2,3),(3,2^2),(2^2,3),(2^3,3)\}$. Finally, using~\cite{BHR-D}
we can determine  the $\CC_6$-subgroups of $\GL(2,3)$, $\GL(3,4)$, $\GL(4,3)$
and $\GL(8,3)$; in all these cases, the bound~\eqref{E3} holds strictly.
This concludes the $\CC_6$ case.  \eoc

From Cases 1--7 we may assume that $H\not\in\CC_1\cup\cdots\cup\CC_7$.
We now define groups $X$ and $Y$ satisfying $X\leqslant H\leqslant Y$ depending on the
nature of the form preserved by $H$. In Case~8 we treat the case where $X\leqslant H\leqslant Y$.
\begin{enumerate}[{\rm (a)}]
\item $H$ preserves no non-degenerate Hermitian, alternating or quadratic
  form on $V$, up to scalar multiplication. Here $d\geqslant2$, $X=\SL(d,p^f)$
  and $Y=\GL(d,p^f)$.
\item $H$ preserves, a non-degenerate
  Hermitian form on $V$ modulo scalars. Here $d>2$, $f$ is even\footnote{Some authors use the notation $\SU(d,p^{f/2})$  and $\GU(d,p^{f/2})$ writing the square root of the field size.}, $X=\SU(d,p^f)$
  and $Y=\C_{p^f-1}\circ\GU(d,p^f)$.
\item $H$ preserves, modulo scalars, a non-degenerate
  alternating form but no non-degenerate quadratic form on $V$. Here $d\geqslant4$
  is even, $X=\Sp(d,p^f)$ and $Y=\C_{p^f-1}\circ\GSp(d,p^f)$.
\item $H$ preserves, modulo scalars, a non-degenerate
  quadratic form on $V$. Here $d>2$, $X=\Omega^\varepsilon(d,p^f)$,
  where $\varepsilon=\pm$ if $d$ is even and $\varepsilon=\circ$ if $d$
  is odd and $Y=\C_{p^f-1}\circ\textup{GO}(d,p^f)$. If $d$ is odd we
  additionally assume that $q$ is odd, since  $X$ is irreducible on $V$.
\end{enumerate}

{\sc Case 8. $X\leqslant H\leqslant Y$.}\quad
First we consider the case where the derived group $X'$ modulo scalars is not
a nonabelian simple group.  Then $X$ is one of: (i) $\SL(2,2)$, (ii) $\SL(2,3)$,
(iii) $\SU(3,2^2)$, (iv) $\Omega(3,3)$, or (v) $\Omega^+(4,p^f)$.
(i)~If $X=\SL(2,2)$, then $Y=\GL(2,2)$ is soluble and the upper bound
of \eqref{E3} holds strictly if $H<Y$, and exactly when $H=Y=L_0$.
(ii)~If  $X=\SL(2,3)$, then $Y=\GL(2,3)$ is
soluble, so $c(H)\leqslant c(\GL(2,3))=5$ which is strictly less than the
upper bound in \eqref{E3}.
(iii) If $X= \SU(3,2^2)$, then $|Y|=|\GU(3,2^2)|=2^3\cdot 3^4$
and hence $c(H)\leqslant c(Y)=7$, and \eqref{E3} holds strictly as
$7< \frac{23}{3}$.  (iv) If $X= \Omega(3,3)$, then
$|Y|=|\{\pm 1\}\times \SO(3,3)|=2^4\cdot3$ and hence $c(H)\leqslant c(Y)=5$,
again, \eqref{E3} holds strictly as $5< 8\log_23-\frac{13}{3}$.
(v) Finally, suppose that $X=\Omega^+(4,p^f)$, and note that $X$ modulo
scalars is the direct product $\overline{X} :=\PSL(2,p^f)\times \PSL(2,p^f)$,
and $H\leqslant \C_{p^f-1}.\overline{X}.\D_8$  if $p$ is odd, and
$H\leqslant \C_{p^f-1}.\overline{X}.\C_2$ if $p=2$. If $p^f=2$ or $3$, then
$H\leqslant (\sym_3\times \sym_3).\C_2$ or $\C_2.(\alt_4\times \alt_4).\D_8$,
and hence $c(H)$ is at most  $5$ or $10$, respectively. In each case
this is strictly less than the upper bound in \eqref{E3}. Suppose now
that $p^f\geqslant4$. Since $H$ contains $X$, we have
$c(H)\leqslant c(\C_{p^f-1}) + c(\overline{X}) + c(\D_8) <   f\log_2p + 5$.
This expression is less than
$\left(\frac83\log_2 p-1\right)(4f)- \log_2 f-\frac43$ for all $p^f\geqslant4$.

For all the other cases $X$, modulo scalars, is a nonabelian simple
group $\overline{X}$ and using information from \cite{KL}*{Table 2.1.C}
about $\Out(\overline{X})$, we see that
$c(H) < 2f\log_2p +\log_2f + 3$, and this is at most
$\left(\frac83\log_2 p-1\right)df- \log_2 f-\frac43$ if and only if
\begin{equation*}
  \frac{2\log_2f + 13/3}{f} \leqslant \frac{8d-6}{3}\log_2p - d.  
\end{equation*}
The left side is at most $13/3$ (taking $f=1$) while the right side is at least $(5d-6)/3$ (taking $p=2$), and $13/3\leqslant (5d-6)/3$ provided $d\geqslant4$. Similarly considering the value of the right side for $p=3$, we see that the inequality holds for all $d\geqslant3$ when $p\geqslant3$. This leaves the cases $d=2$ and $(d,p)=(3,2)$. Suppose first that $(d,p)=(3,2)$. The inequality is easily  seen to hold for $f\geqslant3$. Thus $f\leqslant2$ and $H= X = \SL(3,2^f)$ or $\SU(3,4)$,  so $c(H)\leqslant3$ and the bound~\eqref{E3} holds strictly. 
Finally let $d=2$, so $p^f\geqslant4$. Here $c(H)<  2f\log_2p + 1$, and this is strictly less than the upper bound in \eqref{E3} for all $p^f\geqslant4$. This concludes the proof in Case~8. \eoc

We are now in the case where $X\not\leqslant H\leqslant Y$ where the subgroups
$X$ and $Y$ are defined in the preamble to Case~8.  We apply the
Aschbacher classification~\cite{Asch} of subgroups of $Y$ which do not
contain $X$.  Given our analysis above, and by the definition of $X$, the
only remaining possibility is $H$, modulo scalars, is almost simple; and
its quasisimple normal subgroup $S$ is absolutely irreducible and primitive
on $V$. This case is called type $\CC_9$.

{\sc Case 9. $H\in\CC_9$.}\quad
In this final case, $H$ has a quasisimple normal subgroup $S$ which is
absolutely  irreducible and primitive on $V$. Thus, if  $Z$ is the
subgroup of scalars in $H$ and $T := S/(S\cap Z)$, then
$T\lhdeq H/Z \leqslant \Aut(T)$ and $T$ is a nonabelian simple group.
Therefore $c(H)=c(S\cap Z)+c(T)+c(H/S) < f\log_2 p + 1 + c(\Out(T))$.

Suppose first that $c(\Out(T))\leqslant 2$ and thus $c(H) < f\log_2 p + 3$. It is not hard to show that this is less than the upper bound in $\eqref{E3}$, unless $(d,p^f)=(3,2)$ or $d=2$. Since $\GL(3,2)$ contains no $\CC_9$-subgroups, we must have $d=2$. Here $T=\alt_5$, $p^f\geqslant4$, and $\Out(T)=\C_2$. Thus  $c(H) < f\log_2 p + 2$, and this  is less than the upper bound in $\eqref{E3}$.

We may therefore assume that $c(\Out(T))\geqslant 3$. By considering the possibilities for $T$ when $d=3$, we can then exclude the case $d=3$ and so we assume that $d\geqslant 4$.
The fact that $c(\Out(T))\geqslant 3$ implies (using the  classification of finite simple groups) that $T$ is a  simple group of Lie type (keeping in mind the exceptional isomorphism $\alt_6\cong\PSL(2,9)$). It follows from~\cite{Lieb}*{Theorem 4.1} that $|H/Z|\leqslant p^{3fd}$. By Lemma~\ref{lem:out}(b),  $|\Out(T)|\leqslant \log_2|T|$, and so $|\Out(T)|\leqslant 3fd\, \log_2p$ which yields 
\[
c(H) < f\log_2p + 1 +  \log_2(3df) + \log_2\log_2p.
\]
Suppose that the following inequality holds
\begin{equation}\label{newnewnew}
f\log_2p + 1 +  \log_2 3+\frac{df}{2} + \log_2\log_2p\leqslant  \left(\frac83\log_2 p-1\right)df-\frac {f}{2}-\frac43.
\end{equation}
Using~\eqref{newnewnew} and the fact that $\log_2 f\leqslant \frac {f}2$ for $f\ne3$, we see that, for $f\ne3$,
\begin{align*}
  c(H)<f\log_2p + 1 +  \log_2 3+\frac{df}{2} +  \log_2\log_2p
    &\leqslant \left(\frac83\log_2 p-1\right)df-\frac {f}{2}-\frac43\\
    &\leqslant \left(\frac83\log_2 p-1\right)df-\log_2 f-\frac43,
\end{align*}
as required. Consider the case when $f=3$. Since $d\geqslant4$ and $p\geqslant2$ we have
\[
  \frac73+2\log_23\leqslant 11\leqslant\left(\frac{7d}{2}-3\right)\log_2 p -  \log_2\log_2p
    +\frac{9d}{2}\left(\log_2 p-1\right).
\]
Rearranging gives the desired bound
\[
  c(H)<3\log_2 p+1+\log_23+\frac{3d}{2} + \log_2\log_2p
  \leqslant\left(\frac83\log_2 p-1\right)3d-\log_23-\frac43.
\]

Thus it remains to assume the opposite of~\eqref{newnewnew}. This is equivalent to 
\begin{equation}\label{newnewnew2}
 0> \frac83df\log_2 p-f\log_2p  -\frac{3df}{2} -\frac{f}{2}-\frac73-\log_2 3 -  \log_2\log_2p.
\end{equation}
For fixed $d$ and $f$, the right side of~\eqref{newnewnew2} is an increasing function of $p$. Similarly, for fixed $f$ and $p$, the right side of~\eqref{newnewnew2} is an increasing function of $d$. Setting $p=2$ and $d=4$ shows
\begin{align*}
  &\frac83df\log_2 p-f\log_2p  -\frac{3df}{2} -\frac{f}{2}-\frac73-\log_2 3 -  \log_2\log_2p\\
  &\geqslant \frac{32f}3-f -6f -\frac{f}{2}-\frac73-\log_2 3\\
  &=\frac{19f}{6}-\frac73-\log_2 3.
\end{align*}
However, $\frac{19f}{6}-\frac73-\log_2 3\geqslant0$ for $f\geqslant2$ contrary
to~\eqref{newnewnew2}. Hence the solutions to~\eqref{newnewnew2} have $f=1$.
It is easy to check that~\eqref{newnewnew2} is satisfied for
$(d,p,f)=(4,2,1)$ only. One can then check that, if
$H\leqslant \GL(4,2)\cong \alt_8$, then $c(\Out(T))\leqslant 2$. This case was
handled earlier.
This final argument completes the proof of both the $\CC_9$ case, and
Theorem~\ref{C3}. 
\end{proof}

\section{Proofs of Theorems~\ref{C1} and~\ref{theo:na}}
\label{sec:last}
\begin{proof}[Proof of Theorem~\ref{theo:na}]
  The degree of a non-affine primitive group is at least $5$, and if $G$ is
  such a group of degree $5$ then $G=\alt_5$ or $\sym_5$, $c(G)\leqslant 2$,
  and the bound  $c_{\rm na}\log_2n -\frac{4}{3}$ equals
  $\frac{10}{3\log_25} \log_25 - \frac{4}{3} = 2$. Thus
  $c(G)\leqslant c_{\rm na}\log_2n-\frac{4}{3}$ holds with equality if and only if
  $G=\sym_5$, so the result holds if $n=5$. 

  Assume now that $n>5$ and that  Theorem~\ref{theo:na} holds for groups
  of degree less than~$n$. Let $G$ be a non-affine primitive permutation
  group of degree~$n$.  We first treat the cases: almost simple, and
  simple diagonal (which includes both types {\sc HS} and {\sc SD}
  in the type  descriptions in \cite{P}*{Section 3}). Then we treat
  all other cases together since in these remaining cases $G$ is
  contained in a wreath product in product action. (This includes
  types  {\sc HC, TW, CD} and {\sc PA} as described in \cite{P}*{Section 3}.)

  {\sc Case 1. Almost Simple.}\quad In this case, $T\lhdeq G\leqslant\Aut(T)$,
  where $T$ is a nonabelian simple group. Suppose first that $T=\alt_6$.
  Then $c(G)\leqslant3$. If $n\geqslant10$, then it follows that
  $c_{\rm na}\log_2n -\frac{4}{3} = \frac{10}{3\log_25} (\log_25+1) -\frac{4}{3} >3\geqslant c(G)$.
  On the other hand if $n<10$ then $n=6$, $c(G)\leqslant 2$ (since then
  $G\leqslant\sym_6$), and  $c_{\rm na}\log_2n-\frac{4}{3}> 2$. Thus
  Theorem~\ref{theo:na} holds with strict inequality if $T=\alt_6$.
  Suppose now that $T\ne\alt_6$. Then by Lemma~\ref{lem:out},  
\begin{align*}
c(G)&\leqslant 1+ c(\Out(T)) \leqslant 1+ \log_2|\Out(T)|< 1+\log_2(n/2)=\log_2 n.
\end{align*}
If $n\geqslant11$ then
$\log_2n < 1.4\log_2n-\frac{4}{3} < c_{\rm na}\log_2n-\frac{4}{3}$, as required.
So assume that $6\leqslant n\leqslant 10$. For these degrees the possible almost
simple groups are known and in all case  $c(G)\leqslant 2$ (since $T\ne\alt_6$),
which is strictly less than $ c_{\rm na}\log_2n-\frac{4}{3}$.

{\sc Case 2. Simple Diagonal.}\quad In this case, the socle $N$ of $G$
(the product of the minimal normal subgroups) has the form $N=T^k$,
where $T$ is a nonabelian simple group,  $k\geqslant 2$, and $n=|T|^{k-1}$.
Further  $G/N$ is isomorphic to a subgroup $H$ of $\Out(T)\times \sym_k$.
Thus $H$ is a subdirect subgroup of $H_1\times H_2$, for some
$H_1\leqslant\Out(T)$ and $H_2\leqslant \sym_k$. Furthermore, either $H_2$ is a
transitive subgroup of $\sym_k$, or $k=2$ and $H_2=1$ (for types
{\sc HS} and {\sc SD} respectively).   In either case,
$c(H_2)\leqslant \frac43(k-1)$, by Theorem~\ref{C2}, and $c(H)\leqslant c(H_1)+c(H_2)$
by Lemma~\ref{La}(b). Moreover, by Lemma~\ref{lem:out}(a),
$|\Out(T)| < |T|/16$, (since $T$ has a subgroup of order at least $8$ and
hence of index at most $|T|/8$). Thus
\begin{align*}
c(G)&= k+c(H)\leqslant k+c(H_1)+c(H_2)\overset{\ref{C2},\ref{lem:out}}{\leqslant} k+\log_2\left(\frac{|T|}{16}\right)+\frac43(k-1)\\
    &=\frac73k - 4 + \log_2|T|  - \frac43 = \frac{7k-12}{3} + \log_2 |T| - \frac43.
\end{align*}
Since $c_{\rm na} > 1.4$, it is sufficient to prove that this expression is at most 
\[
1.4 \log_2 n - \frac43 = 1.4 (k-1)\log_2|T| -\frac43 =  \frac{7k-12}{5}\log_2|T|    + \log_2 |T| -\frac43.
\]
This is true since $\log_2|T| \geqslant 5/3$. Hence the bound holds strictly.

{\sc Case 3. Product Action.}\quad In this final case  $n=a^b$ with $b\geqslant2$, and $G$ is a subwreath subgroup of $A\wr B$, where  $B$ is a transitive permutation group of degree $b$ and $A$ is a primitive permutation  group of degree $a$. Moreover, either $A$ is  almost simple (for $G$ of type {\sc PA}),  or of simple diagonal type (namely of type {\sc SD}  if $G$ has type {\sc CD}, and of type {\sc HS} if $G$ has type {\sc HC} or {\sc TW}). In either of these cases, it follows by induction and Cases~1 and~2 that $c(A) \leqslant c_{\rm na} \log_2 a-\frac43$. Also, by Lemma~\ref{La}(c), $c(G)\leqslant b\cdot c(A)+c(B)$. Thus by Theorem~\ref{C2},
\begin{equation*}
c(G)\leqslant b\cdot c(A)+c(B)\leqslant b\left(c_{\rm na}\log_2a-\frac43\right)+\frac43(b-1)= b c_{\rm na}\log_2a-\frac43 =c_{\rm na}\log_2 n-\frac43,
\end{equation*}
and equality holds if and only if all of the following hold: 
\[
c(A)= c_{\rm na}\log_2a-\frac43,\quad  c(B)=\frac43(b-1)\quad \mbox{and}\quad  c(G)= b\cdot c(A)+c(B).
\] 
By induction,  Theorem~\ref{C2}, and Lemma~\ref{La}(c), it follows that $G=A\wr B$, $b=4^k$  and $B=T_k$ for some $k\geqslant1$ (since $b>1$), and $A$ is one of the groups listed in  Theorem~\ref{theo:na}. Since $A$ is  almost simple or of simple diagonal type it follows that $a=5$ and $A=\sym_5$. Thus $n=5^{4^k}$ and $G=\sym_5\wr T_k$ in product action. This completes the proof of  Theorem~\ref{theo:na}.
\end{proof}

\begin{proof}[Proof of Theorem~\ref{C1}]
  For $n\leqslant4$, the result can be checked by inspection. Note that,
  for $n=4$, the bound is met by $G=P_1= \sym_4$.  Henceforth assume
  that $n\geqslant 5$, that $G$ is a primitive permutation group of degree $n$.
  If $G$ is non-affine then, by Theorem~\ref{theo:na},
  $c(G)\leqslant c_{\rm na}\log_2n -\frac43$. Since $c_{\rm na} < \frac83$,
  Theorem~\ref{C1} holds with a strict inequality in this case.

  Thus we may assume that $G$ is of affine type, so  $n=p^d$ for some
  prime $p$ and integer $d\geqslant1$, and  $G=(\C_p)^d\rtimes H$ where  $H$
  is an irreducible subgroup of $\GL(d,p)$. Thus by Theorem~\ref{C3},
  $c(H)\leqslant (\frac83\log_2 p-1)d-\frac43$, and  therefore
\begin{equation*}
  c(G)=d+c(H)\overset{\ref{C3}}{\leqslant} d+\left(\frac83\log_2 p-1\right)d-\frac43=\frac{8d}3\log_2 p-\frac43=\frac{8}3\log_2n-\frac43.
\end{equation*}
Moreover, by Theorem~\ref{C3} (and since $H$ is an irreducible linear group over the field $\F_p$), equality occurs if and only if $p=2$, $d=2^{2k+1}=2\cdot4^k$ for some $k\geqslant0$, and $H$ is linearly isomorphic to $L_k$
(recall $n=p^d\ne 4$). Thus $n=2^d = 4^{4^k}$ with $k\geqslant1$, and 
\[
  G=\C_2^{2\cdot 4^k}\rtimes L_k=(\C_2^2)^{4^k}\rtimes (\GL(2,2)\wr T_k)
  =(\C_2^2 \rtimes \GL(2,2))\wr T_k\cong \sym_4\wr T_k=P_k.\qedhere
\]
\end{proof}

\section{Proof and examples for Theorem~\ref{theo:qpsp}}\label{sec:semi}

\begin{proof}
  Let $G\leqslant\Sym(\Omega)$ with $n=|\Omega|$.

  (a) Suppose first that $G$ is quasiprimitive but not primitive.
  Let $\Delta$ be a system of maximal (proper) blocks of imprimitivity
  for $G$ in $\Omega$, and let $d=|\Delta|$. Then
  $d\leqslant n/2$ and $d\mid n$ since $G$ is imprimitive.
  Also, $G^\Delta$  is primitive as $\Delta$ is maximal. Since $G$ is quasiprimitive, the kernel of the action
  of $G$ on $\Delta$ is trivial, and so $G\cong{}G^\Delta$. Thus $G$ and
    $G^\Delta$ have isomorphic socles. If $G$ were of affine type, then
    $\textrm{soc}(G)$ would be abelian and regular. As
    $\textrm{soc}(G^\Delta)$ is abelian, it is regular on $\Delta$. This proves
    that $n=|\textrm{soc}(G)|=|\textrm{soc}(G^\Delta)|=d$, a contradiction.
  Thus $G$ is non-affine, and so, by Theorem~\ref{theo:na}, we have the required bound 
\begin{equation*}
c(G)=c(G^\Delta) 
 \leqslant{}c_{\rm na}\log_2{d}-\frac{4}{3}
    \leqslant{}c_{\rm na}\log_2{\frac{n}{2}}-\frac{4}{3}
    =c_{\rm na}(\log_2{n}-1)-\frac43.
\end{equation*}
 
(b)~Let $G\leqslant\Sym(\Omega)$ be semiprimitive but not quasiprimitive.
As $G$ is not quasiprimitive, $G$ must have a nontrivial intransitive
normal subgroup. Let $M$ be a maximal such normal subgroup of $G$,
let $\Sigma$ be the set of $M$-orbits, and let $m=|M|$. Since $G$
is semiprimitive and $M$ is intransitive, we have $G^\Sigma\cong G/M$
by \cite{BM}*{Lemma 2.4}. 
We now show that $G^\Sigma$ is quasiprimitive.
Suppose $N^\Sigma\lhdeq G^\Sigma$ where $N\lhdeq G$. If $N\leqslant M$, then
$N^\Sigma$ is trivial. If $N\not\leqslant M$, then $M<NM\lhdeq G$, and by
the maximality of $M$, $NM$ is transitive on $\Omega$, and hence
$N^\Sigma$ is transitive on $\Sigma$. Therefore $G^\Sigma$ is
quasiprimitive. Using the argument in the previous paragraph, $G^\Sigma$ is isomorphic
to a primitive permutation group of degree dividing $|\Sigma|$.

Since $M$ is an intransitive normal subgroup of the semiprimitive
group $G$, $M$ is semiregular, and hence $|\Sigma|=|\Omega|/|M|=n/m$.
By the previous paragraph, $G^\Sigma$ is isomorphic to a primitive permutation group of
degree $r$ dividing $n/m$. By Theorem~\ref{C1}, 
$c(G^\Sigma)\leqslant{}\frac{8}{3}\log_2 r -\frac{4}{3}$.
As $|M|=m\geqslant 2$, the desired bound is proved as~follows
\begin{align*}    c(G)&=c(G/M)+c(M)=c(G^\Sigma)+c(M)\leqslant{}\frac{8}{3}\log_2 r-\frac{4}{3}+c(M)\\
    &\leqslant{}\frac{8}{3}\log_2\left(\frac{n}{m}\right)-\frac{4}{3}+\log_2{m}
     ={}\frac{8}{3}\log_2{n}-\frac{4}{3}-\frac{5}{3}\log_2{m}
     \leqslant\frac{8}{3}\log_2{n}-\frac{9}{3}.\qedhere
\end{align*}
\end{proof}

\begin{remark}\label{rem:semiprim}
 We claim that, if equality holds in Theorem~\ref{theo:qpsp}(b), 
 then $n$ is a power of $2$, $G$ 
  is a $\{2,3\}$-group  (and hence soluble), and each plinth of $G$ is a $2$-group
   (see the definition after Theorem~\ref{theo:qpsp} and~\cite{Semiprimitive}*{Corollary~3.11}). 
    Suppose that equality holds in Theorem~\ref{theo:qpsp}(b) and hence
  in the displayed equation above. 
  To start with, this means that $r=n/m$ and hence 
  that $G^\Sigma$ is a primitive permutation group of degree $n/m$. 
Moreover equality must hold in Theorem~\ref{C1} for $G^\Sigma$. Thus $G^\Sigma=P_k$
  for some $k$ and $G^\Sigma$ is of affine type where $n/m=4^{4^k}$.
  Furthermore equality holding implies that  $c(M)=\log_2m$, so $M$ is
  a 2-group and $m$ is a $2$-power.  Thus  $n$ is a $2$-power, and $G$ 
  is a (soluble) $\{2,3\}$-group.  Let $K$ be an arbitrary plinth of $G$, 
  and let $L$ be a normal subgroup of $G$, properly contained in $K$, 
  and maximal respect to these properties. By the definition of a plinth, $L$ is 
  intransitive and hence semiregular. Hence $L$ is a $2$-group since $n$ is a 
  $2$-power. Also it follows from the maximality of $L$ that $K/L$ is a transitive 
  minimal normal subgroup of $G/L$, and acts faithfully on the set of, say $r$, $L$-orbits
  in $\Omega$. Now $r = n/|L|$, and so $r$ is a $2$-power. Then since $G/L$ is 
  soluble, its transitive minimal normal subgroup $K/L$ must be an elementary 
  abelian group of $2$-power order $r$. Hence $K$ is a $2$-group, proving the claim.
 \end{remark}

\begin{example}\label{ex:semiprim}
  We will construct infinitely many groups $H_0, H_1,\dots$, for which
  the bound in Theorem~\ref{theo:qpsp}(b) is attained.
  
Consider $\GL(2,3)$ as a permutation group of degree $8$ on the set $\Delta$ of non-zero vectors of $\mathbb{F}_3^2$. Let $k\geqslant 0$ and let $H_k=\GL(2,3) \wr T_k$ act in its  product action on $\Delta^{4^k}$. Let $B_k$ be the base group of $H_k$ (so that $H_k=B_k\rtimes T_k$), and let $Z_k$ be the center of $B_k$. As $T_k$ has degree $4^k$, we have $B_k\cong \GL(2,3)^{4^k}$ and $Z_k\cong \C_2^{4^k}$. View $Z_k$ as a vector space over the field $\F_2$ with basis consisting of the  generators of the $4^k$ copies of $Z(\GL(2,3))$. Let $N_k$ be the codimension 1 subspace of $Z_k$ comprising vectors with coordinates summing to zero in $\F_2$. Note that $N_k$ is an intransitive normal subgroup of~$H_k$. 

Since $\GL(2,3)$ is semiprimitive on $\Delta$, $H_k$ is semiprimitive on $\Delta^{4^k}$, by~\cite{Semiprimitive}*{Theorem 9.7}. Hence $N_k$ is semiregular on $\Delta^{4^k}$. It follows by~\cite{Semiprimitive}*{Lemma 3.1} that $H_k/N_k$ acts faithfully and semiprimitively on the set $\Omega$ of $N_k$-orbits in $\Delta^{4^k}$. Here $H_k/N_k$ has degree 
\[
  |\Omega|=\frac{8^{4^k}}{|N_k|}=\frac{(2\cdot 4)^{4^k}}{2^{4^k-1}}=2\cdot 4^{4^k}
\]
while
\[
  c(H_k/N_k)=c(H_k)-c(N_k)=5\cdot 4^k+\frac{4}{3}(4^k-1)-(4^k-1)=\frac{16}{3}4^k-\frac{1}{3}=\frac{8}{3}\log_2 (2\cdot 4^{4^k})-3,
\]
as in Theorem~\ref{theo:qpsp}(b). Note also that $H_k/N_k$ is not
quasiprimitive on $\Omega$ since it has a normal subgroup $Z_k/N_k$
of order $2$ and, as $|\Omega|>2$,  $Z_k/N_k$ is intransitive on $\Omega$.
\end{example}

\begin{remark}\label{rem:qp}
  We show that the bound in Theorem~\ref{theo:qpsp}(a), is never attained.
  Suppose to the contrary that $G$ is quasiprimitive of degree $n$, but
  not primitive, and that $c(G)=c_{\rm na}(\log_2{n}-1)-\frac43$.
  It follows from the proof of Theorem~\ref{theo:qpsp}(a) that $G$
  acts primitively on a set $\Delta$ of $n/2$ blocks of imprimitivity
  each of size $2$, and that the induced primitive group $G^\Delta$ is not
  affine, and $c(G^\Delta)$ achieves the upper bound of Theorem~\ref{theo:na}.
  Thus $G\cong G^\Delta = \sym_5\wr T_k$ in product action and the stabiliser
  of a block $\delta\in\Delta$ is $G_\delta\cong \sym_4\wr T_k$. Since $G$ is
  quasiprimitive on $n$ points, the stabiliser in $N=(\alt_5)^{4^k}$ of a
  point $\alpha\in\delta$ is a subgroup of index $2$ in
  $N_\delta\cong (\alt_4)^{4^k}$. However, no such subgroup exists.
\end{remark}

In Example~\ref{ex:qp} we provide an infinite family of quasiprimitive groups  $G$ which are not primitive and are such that the composition lengths $c(G)$ grow logarithmically with the degree. The competing requirements for such a construction are (a) to use a simple group such as $\alt_5$ for the direct factors of the socle, and a group $T_k$ permuting the factors of the socle to make $c(G)$ large relative to the degree; and (b)  to define the point stabiliser to ensure that the socle is transitive. 

\begin{example}\label{ex:qp}
Let $k$ be a positive integer, and consider the group $X= \sym_5\wr T_k$, which has a primitive action of degree $d=5^{4^k}$ on a set $\Delta$, and satisfies  $c(X)=c_{\rm na}\log_2{d}-\frac43$, by Theorem~\ref{theo:na}. There is an element $\delta\in\Delta$ such that $X_\delta = \sym_4\wr T_k$ where each factor $\sym_4$ of the base group of $X_\delta$ is the stabiliser in $\sym_5$ of the point $5$.

Let $N = \alt_5^{4^k}$, the unique minimal normal subgroup of $X$, and $B=\sym_5^{4^k}$, the base group of $X$. Let $B_0=\sym_2^{4^k}$ denote the subgroup of $B$ which projects to $\langle(1,2)\rangle$ on each factor $\sym_5$ of $B$, so $B=N\rtimes B_0$. Also $B_0\leqslant G_\delta$ and $B_0$ normalises $N_\delta = \alt_4^{4^k}$.

The transitive conjugation-action of the top group $T_k$ on the $4^k$ factors $\sym_5$ of $B$ preserves a system of imprimitivity with $4^{k-1}$ blocks of size $4$. Let $D = {\rm Diag}(\sym_2^4)$  and let  $M = D^{4^{k-1}}$ be the subgroup of $B_0$ such that the image of $M$ under projection to  $\sym_2^4$ is $D$, for each of the $4^{k-1}$ blocks of size $4$. Then $M$ is $T_k$-invariant, being constant on each minimal block for $T_k$ (of size $4$). 

Define $G$ to be the subgroup $G= N.M.T_k$ of $X$. Since $G$ contains the top group $T_k$, it follows that $N$ is a minimal normal subgroup of $G$, and in fact it is the unique minimal normal subgroup since $C_G(N)=C_X(N)=1$. It is not difficult to see that 
$G_\delta = N_\delta.M.T_k$ is maximal and core-free in $G$, so $G$ acts faithfully and primitively on $\Delta$ of degree $d=5^{4^k}$.

We define a subgroup $H$ of $G_\delta$ such that $G$ acts quasiprimitively on the coset space $\Omega=[G:H]$. Let $O_2(N_\delta) \cong (\C_2^2)^{4^k}$ be the largest normal $2$-subgroup of $N_\delta$, and let   $D_1 = {\rm Diag}(\sym_3^4)$  (with $\sym_3$ fixing points $4, 5$) and  $M_1 = D_1^{4^{k-1}}$, so $M_1$ contains $M$ and $M_1$ is $T_k$-invariant. Let $H= O_2(N_\delta).M_1.T_k$. Then $H$ is a subgroup of $G_\delta$ of index 
\[
|G_\delta:H|  = |N_\delta: H\cap N_\delta| = 3^{4^k - 4^{k-1}} = 3^{3.4^{k-1}}.
\]
Since $G_\delta$ is a core-free subgroup of $G$, so is $H$ and hence $G$ acts transitively and faithfully on $\Omega$.  Moreover, the displayed equation implies that $N$ is transitive on $\Omega$. Since $N$ is the unique minimal normal subgroup of $G$, $G$ is quasiprimitive (but not primitive) on $\Omega$.

The degree is $n=|\Omega| = |\Delta||G_\delta:H| =  5^{4^k}\cdot 3^{3.4^{k-1}} = x^{4^k}$, where $x = 5\cdot 3^{3/4}$. Thus $\log_2 n = 4^k \log_2 x$. Also (using Theorem~\ref{C2}) 
\[
c(G) = c(N) + c(M) + c(T_k) = 4^k + 4^{k-1} + \frac 43(4^k-1)  =  4^k\left ( 1 + \frac 14 + \frac 43 \right) - \frac 43 = \frac{31}{12}\, 4^k  - \frac 43.
\]
It follows that $c(G) =  c \log_2 n -\frac 43 $, where  $c=   \frac{31}{12\log_2 x} = \frac{31}{12\log_25\, +\, 9\log_2 3}  = 0.73\cdots$.
\end{example}

\section*{Acknowledgments}
We would like to thank the Centre for the Mathematics of Symmetry and Computation for its support, as this project grew from a problem posed by the first author at the annual CMSC retreat in 2017. The first two authors gratefully acknowledge the support of the Australian Research Council Discovery Grant DP160102323. The third author thanks the Australian Government for the support of a Research Training Program grant. Finally, we thank the referee for carefully reading our manuscript.


\begin{thebibliography}{99}


\bib{Asch}{article}{
    AUTHOR = {Aschbacher, Michael},
     TITLE = {On the maximal subgroups of the finite classical groups},
   JOURNAL = {Invent. Math.},
    VOLUME = {76},
      YEAR = {1984},
     PAGES = {469--514},
      ISSN = {0020-9910},
       URL = {http://dx.doi.org/10.1007/BF01388470},
}


\bib{AG}{article} {
    AUTHOR = {Aschbacher, Michael},
    AUTHOR = {Guralnick, Robert M.},
     TITLE = {On abelian quotients of primitive groups},
   JOURNAL = {Proc. Amer. Math. Soc.},
    VOLUME = {107},
      YEAR = {1989},
     PAGES = {89--95},
}

\bib{Ba81}{article} {
    AUTHOR = {Babai, L\'aszl\'o},
     TITLE = {On the order of uniprimitive permutation groups},
   JOURNAL = {Ann. of Math. (2)},
    VOLUME = {113},
      YEAR = {1981},
     PAGES = {553--568},
}

\bib{Ba82}{article} {
    AUTHOR = {Babai, L\'aszl\'o},
     TITLE = {On the order of doubly transitive permutation groups},
   JOURNAL = {Invent. Math.},
    VOLUME = {65},
      YEAR = {1981/82},
     PAGES = {473--484},
}
		


\bib{BM}{article}{
 author={\'Aron Bereczky},
 author={Attila Mar\'oti},
     TITLE = {On groups with every normal subgroup transitive or semiregular},
   JOURNAL = {J. Algebra},
    VOLUME = {319},
      YEAR = {2008},
     PAGES = {1733--1751},

}

\bib{BHR-D}{book}{
   author={Bray, John N.},
   author={Holt, Derek F.},
   author={Roney-Dougal, Colva M.},
   title={The maximal subgroups of the low-dimensional finite classical groups},
   series={London Mathematical Society Lecture Note Series},
   volume={407},
   note={With a foreword by Martin Liebeck},
   publisher={Cambridge University Press, Cambridge},
   date={2013},
   pages={xiv+438},
   isbn={978-0-521-13860-4},
}

\bib{Cam81}{article} {
    AUTHOR = {Cameron, Peter J.},
     TITLE = {Finite permutation groups and finite simple groups},
   JOURNAL = {Bull. London Math. Soc.},
    VOLUME = {13},
      YEAR = {1981},
     PAGES = {1--22},
}

\bib{CQR-D}{article}{
   author={Coutts, Hannah J.},
   author={Quick, Martyn},
   author={Roney-Dougal, Colva M.},
   title={The primitive permutation groups of degree less than 4096},
   journal={Comm. Algebra},
   volume={39},
   date={2011},
   pages={3526--3546},
}

\bib{DL}{article}{
    AUTHOR = {Detomi, Eloisa},
    author = {Lucchini, Andrea},
     TITLE = {Invariable generation of permutation groups},
   JOURNAL = {Arch. Math. (Basel)},
    VOLUME = {104},
      YEAR = {2015},
     PAGES = {301--309},
}

\bib{F}{article}{
   author={Fisher, R. K.},
   title={The polycyclic length of linear and finite polycyclic groups},
   journal={Canad. J. Math.},
   volume={26},
   date={1974},
   pages={1002--1009},
}

\bib{F2}{article}{
   author={Fisher, R. K.},
   title={The number of non-solvable sections in linear groups},
   journal={J. London Math. Soc. (2)},
   volume={9},
   date={1974/75},
   pages={80--86},
}


\bib{CompFactors}{article}{
   author={Giudici, Michael},
   author={Glasby, S.\,P.},
   author={Li, Cai Heng},
   author={Verret, Gabriel},
   title={The number of composition factors of order $p$ in completely reducible groups of characteristic $p$},
   journal={J. Algebra},
       VOLUME = {490},
      YEAR = {2017},
     PAGES = {241--255},
}

\bib{Semiprimitive}{article}{
   author={Giudici, Michael},
   author={Morgan, Luke},
   title={A theory of semiprimitive groups},
   Eprint = {arXiv:1603.00187},
}

\bib{GMP}{article} {
    AUTHOR = {Guralnick, Robert M.},
    author = {Mar\'oti, Attila},
    author = {Pyber, L\'aszl\'o},
     TITLE = {Normalizers of primitive permutation groups},
   JOURNAL = {Adv. Math.},
    VOLUME = {310},
      YEAR = {2017},
     PAGES = {1017--1063},
}

\bib{KL}{book}{
    AUTHOR = {Peter Kleidman},
    AUTHOR = {Martin Liebeck},
     TITLE = {The subgroup structure of the finite classical groups},
    SERIES = {London Mathematical Society Lecture Note Series},
    VOLUME = {129},
 PUBLISHER = {Cambridge University Press, Cambridge},
      YEAR = {1990},
     PAGES = {x+303},
      ISBN = {0-521-35949-X},
       URL = {http://dx.doi.org/10.1017/CBO9780511629235},
}



\bib{K}{article}{
   author={Kohl, Stefan},
   title={A bound on the order of the outer automorphism group of a finite simple group of given order},
   Eprint = {https://stefan-kohl.github.io/preprints/outbound.pdf},
     note = {Accessed: 2017-11-14},
}


\bib{Lieb}{article}{
    AUTHOR = {Liebeck, Martin W.},
     TITLE = {On the orders of maximal subgroups of the finite classical groups},
   JOURNAL = {Proc. London Math. Soc. (3)},
    VOLUME = {50},
      YEAR = {1985},
     PAGES = {426--446},
 }

\bib{LuMM}{article}{
    AUTHOR = {Lucchini, A.},
    AUTHOR = {Menegazzo, F.},
	AUTHOR = {Morigi, M.},
     TITLE = {On the number of generators and composition length of finite
              linear groups},
   JOURNAL = {J. Algebra},
    VOLUME = {243},
      YEAR = {2001},
     PAGES = {427--447},
       URL = {http://dx.doi.org/10.1006/jabr.2001.8882},
}
\bib{LP}{book}{
   author={Lux, Klaus},
   author={Pahlings, Herbert},
   title={Representations of groups},
   series={Cambridge Studies in Advanced Mathematics},
   volume={124},
   note={A computational approach},
   publisher={Cambridge University Press, Cambridge},
   date={2010},
   pages={x+460},
   isbn={978-0-521-76807-8},
}

\bib{M}{article}{
   author={Mar\'oti, Attila},
   title={On the orders of primitive groups},
   journal={J. Algebra},
   volume={258},
   date={2002},
   pages={631--640},
   issn={0021-8693},
}


\bib{P}{article} {
    AUTHOR = {Praeger, Cheryl E.},
     TITLE = {The inclusion problem for finite primitive permutation groups},
   JOURNAL = {Proc. London Math. Soc. (3)},
    VOLUME = {60},
      YEAR = {1990},
     PAGES = {68--88},
}

\bib{PS}{article}{
   author={Praeger, Cheryl E.},
   author={Saxl, Jan},
   title={On the orders of primitive permutation groups},
   journal={Bull. London Math. Soc.},
   volume={12},
   date={1980},
   pages={303--307},
}

\bib{PSer}{article}{
    AUTHOR = {Praeger, Cheryl E.},
    AUTHOR = {Seress, \'Akos},
     TITLE = {Probabilistic generation of finite classical groups in odd characteristic by involutions},
   JOURNAL = {J. Group Theory},
    VOLUME = {14},
      YEAR = {2011},
     PAGES = {521--545},
      ISSN = {1433-5883},
       URL = {http://dx.doi.org/10.1515/JGT.2010.061},
}

\bib{Pyb93}{incollection} {
    AUTHOR = {Pyber, L\'aszl\'o},
     TITLE = {Asymptotic results for permutation groups},
 BOOKTITLE = {Groups and computation ({N}ew {B}runswick, {NJ}, 1991)},
    SERIES = {DIMACS Ser. Discrete Math. Theoret. Comput. Sci.},
    VOLUME = {11},
     PAGES = {197--219},
 PUBLISHER = {Amer. Math. Soc., Providence, RI},
      YEAR = {1993},
}

\bib{Wie69}{book}{
   author={Wielandt, Helmut},
   title={Permutation groups through invariant relations and invariant functions},
   series={Lecture Notes},
   publisher={Ohio State University, Columbus},
   date={1969},
}


\end{thebibliography}
\end{document}